\documentclass[12pt]{article}
\usepackage[utf8]{inputenc}
\usepackage{cite}
\usepackage{amsthm}
\usepackage{amsmath}
\usepackage{amssymb}
\usepackage{yfonts}
\usepackage[colorlinks]{hyperref}
\usepackage{authblk}
\usepackage{graphicx}
\usepackage{multicol}
\usepackage{comment}
\usepackage{blkarray}
\newtheorem{thm}{Theorem}[section]
\newtheorem{lem}[thm]{Lemma}

\theoremstyle{definition}
\newtheorem{claim}{Claim}[section]
\newtheorem{defn}{Definition}[section]
\newtheorem{conj}{Conjecture}[section]

\theoremstyle{remark}
\newtheorem{rem}{Remark}

\begin{document}

\title{The variety defined by the matrix of diagonals is $F$-pure}
\author{Zhibek Kadyrsizova}
\affil{{\small Department of  Mathematics, Nazarbayev University \newline
zhibek.kadyrsizova@nu.edu.kz}}
\date{}
\maketitle
\begin{abstract}
We prove that the variety defined by the determinant of the matrix formed by the diagonals of powers of a given matrix  is $F$-pure for matrices of all sizes and in all positive prime characteristics. Moreover, we find a homogeneous system of parameters for it. 
\end{abstract}
\section{Introduction and preliminaries}

Let $X=(x_{ij})_{1\leq i, \, j \leq n}$ be a square matrix of size $n$ with indeterminate entries over a field $K$ and $R=K[X] $ be the polynomial ring over $K$ in $\{x_{ij}\, |\, 1\leq i, \, j \leq n\}$. Let $D(X)$ be an $n\times n$ matrix whose $j$th column consists of the diagonal entries of the matrix $X^{j-1}$ written from left to right with the convention that $X^0$ is the identity  matrix of size $n$. Define $\mathcal{P}(X)=\det(D(X))$.  
In \cite{Y11} and \cite{You20}, H-W.Young studies the varieties of nearly commuting matrices and derives their important properties such as the decomposition into irreducible components through the use of the polynomial $\mathcal{P}(X)$. We hope to understand better their  Frobenius singularities and as it can be seen in \cite{K18} this is also closely related to the singularities of the variety defined by $\mathcal{P}(X)$. In this paper we prove that the latter is $F$-pure and find a homogeneous system of parameters for it. Let us first define the necessary preliminaries. 
The following notation is fixed for the rest of the paper.
Let \[X=\left[\begin{array}{cccc}
                                                                 x_{11}&x_{12}&\ldots&x_{1n}\\
                                                                 &\ldots&\ldots&\\
                                                                 x_{n1}&x_{n2}&\ldots&x_{nn} \end{array}\right] \]and \[\tilde{X}=\left[\begin{array}{c|c}
                                                                                                                                                                                                 X_0&0\\
                                                                                                                                                                                                 \hline
                                                                                                                                                                                                0&x_{nn} \end{array}\right] \]
                                   where \[X_0=\left[\begin{array}{cccc}
                                                                 x_{11}&\ldots&x_{1, n-1}\\
                                                                 &\ldots&\\
                                                                 x_{n-1,1}&\ldots&x_{n-1,  n-1} \end{array}\right] .\] 
That is, $\tilde{X}$ is the matrix obtained from $X$ by setting  all the entries of its last column and the last row  to 0 except for the entry $x_{nn}$. Let  $X_0$ be the matrix obtained from $X$ by deleting its last column and last row.
\begin{lem}[\cite{You20}, Lemma 1.3] $\mathcal{P}(X)$ is an irreducible polynomial. \end{lem}
\begin{lem}[\cite{You20}, Proof of Lemma 1.3] \label{ColP}  $\mathcal{P}(\tilde{X})=\mathcal{P}(X_0)c_{X_0}(x_{nn})$ where $c_{X_0}(t)$ is the characteristic polynomial of $X_0$.  \end{lem}

\begin{rem} The above lemma is also true when we annihilate either only the last column or only the last row of $X$ with the exception of the entry $x_{nn}$. \end{rem}  

\section{A system of parameters}
First, we find a homogeneous system of parameters on $R/(\mathcal{P}(X))$.  To do so we prove several useful lemmas. 
\begin{lem} \label{Powers} Let $A$ be a square matrix of size $n$ with integer entries and with $\det(A)=\pm 1$. Then  the following are equivalent
\begin{enumerate}
\item[(a)] $\mathcal{P}(A)=\pm 1$.
\item[(b)] The diagonals of $I, \, A, \, A^2, \ldots, A^{n-1}$ span $\mathbb{Z}^n$. 
\item[(c)] The diagonals of the elements of $\mathbb{Z}[A, A^{-1}]$ span $\mathbb{Z}^n$. 
\item[(d)] There exist $n$ integer powers of $A$ with the property that their diagonals span $\mathbb{Z}^n$.
\end{enumerate}
\end{lem}
\begin{proof}
The equivalence of (a) and (b) is clear as the columns of the matrix $D(A)$ are the diagonals of $I, \, A, \, A^2, \ldots, A^{n-1}$. 

By Cayley-Hamilton's theorem we have that $A^n \in \sum_{i=0}^{n-1}\mathbb{Z}A^i$ and hence $I \in \sum_{i=1}^{n}\mathbb{Z}A^i$. Therefore, since $A$ is invertible, we also have that $A^{-1} \in \sum_{i=0}^{n-1}\mathbb{Z}A^i$ as well as $A^{-h}$ for all integers $h$. Thus $\mathbb{Z}[A]=\mathbb{Z}[A, A^{-1}]$. 
Finally, \[\det\left[\text{diag}(I) \, \text{diag}(A)\,  \ldots \text{diag}(A^{n-1})\right]=\pm 1 \text{ if and only if}\]
\[\text{diag}(I), \text{diag}(A), \ldots, \text{diag}(A^{n-1}) \text{ span } \mathbb{Z}^n \text{ if and only if}\]
the diagonals of all the integer powers of $A$ span $\mathbb{Z}^n$ if and only if there exist $n$ integer powers of $A$ with the property that their diagonals span $\mathbb{Z}^n$. 

\end{proof}
\begin{lem} \label{P(A)=1}
Let \[A= \left[\begin{array}{ccccc}
                                      1&1&\ldots&1&1\\
                                      1&1&\ldots&1&0\\
                                      &\ldots&\ldots&&\\
                                     1&0&\ldots&0&0\\ \end{array}\right] .\] be a square matrix of size $n\geq 1$ with the property that all the entries  strictly below the main anti-diagonal are 0 and the rest are equal to 1. Then $\mathcal{P}(A)$ is equal to either 1 or -1. \end{lem}

\begin{proof}
First observe that $\det(A)=\pm1$. Hence the matrix is invertible and it can be shown that 
\[B=A^{-1}= \left[\begin{array}{ccccc}
                                      0&0&\ldots&0&1\\
                                      0&0&\ldots&1&-1\\
                                      &\ldots&\ldots&&\\
                                     1&-1&\ldots&0&0\\ \end{array}\right] .\]
$B$ has two non-zero anti-diagonals and the rest of the entries are equal to 0.                                      
To show that $\mathcal{P}(A)=\pm1$ it is necessary and sufficient to show that there exist $n$ powers of $B$ so that their diagonals span $\mathbb{Z}^n$, see Lemma~\ref{Powers}. 
We claim that for this purpose it is sufficient to take $n$ odd powers of $B$.   
\begin{claim} $$B^2= \left[\begin{array}{cccccc}
                                      1&-1&0&\ldots&0&0\\
                                      -1&2&-1&\ldots&0&0\\
                                      &\ldots&\ldots&&\\
                                      0&0&0&\ldots&2&-1\\
                                     0&0&0&\ldots&-1&2\\ \end{array}\right] .$$\end{claim}  
                                     
We show this by induction with the induction step equal to 2. Cases $n=2$ and $n=3$ can be easily verified. \\

Write $B$ as                                                                   
\begin{small}    $$\left[\begin{array}{c|c|c}
  0& 0\ldots 0&1 \\ \hline
  0& &-1\\
  0 & \raisebox{-15pt}{{\large \mbox{{$B_0$}}}}&0\\[-4ex]
  \vdots && \vdots \\
  0&&0\\
  \hline
  1 &-1\,0 \ldots 0&0\\ \end{array} \right] $$\end{small}
  
then  \begin{small}    $$B^2=\left[\begin{array}{c|c|c}
  1& -1\quad 0\ldots\ldots\ldots 0 \quad 0&0 \\ \hline
  -1& &0\\
  0& \raisebox{-15pt}{{\large \mbox{{$B_0^2+\left[\begin{array}{cccc}
                                      1&0&\ldots&0\\
                                      0&0&\ldots&0\\
                                     0&0&\ldots&0\\ \end{array}\right] $}}}} &0 \\[-4ex]
  \vdots && \vdots \\
  0&&-1\\
  \hline
  0 &0\quad 0 \ldots\ldots \ldots 0\, \,  -1&2\\ \end{array}\right]. $$     \end{small}

\begin{claim} For all $ 1\leq j \leq n-1$,   we have that 
\[(B^{2j-1})_{kl}= \left\{\begin{array}{cc}
                                     (-1)^{j+1}& \text { if } k+l=n-j+2, \\ 
                                     (-1)^j & \text { if } k+l=n+j+1,\\
                                     0& \text { if } k+l\leq n-j+1 \text{ or } k+l \geq n+j+2. \end{array}\right.\]     
                                      \end{claim}   
We prove the claim by induction on $j$. When $j=1$, it is true.
Suppose that the claim is true for all integers less than or equal to $j$. 
Then \[B^{2(j+1)-1}=B^{2j+1}=B^2B^{2j-1}.\]
Therefore, \[(B^{2j+1})_{kl}=(B^2B^{2j-1})_{kl}=\sum_{s=1}^{n}(B^2)_{ks}(B^{2j-1})_{sl}=\]
For now assume that $k>1$.
\[(B^2)_{kk}(B^{2j-1})_{kl}+(B^2)_{k,k+1}(B^{2j-1})_{k+1,l}+(B^2)_{k,k-1}(B^{2j-1})_{k-1,l}.\]
Suppose first that $k+l=n-j+1$. Then  
\[(B^{2j+1})_{kl}=(B^2)_{kk}\cdot 0+(-1)(-1)^{j+1}+(-1)\cdot 0=(-1)^{j+2}. \]
Next, suppose that  $k+l=n+j+2$
\[(B^{2j+1})_{kl}=(B^2)_{kk}\cdot 0+(-1)\cdot 0+(-1)(-1)^{j}=(-1)^{j+1}. \]
Finally, consider the case when $k+l\leq n-j \text{ or } k+l \geq n+j+3$. We have that 
\[(B^{2j+1})_{kl}=(B^2)_{kk}\cdot 0+(-1)\cdot 0+(-1)\cdot 0=0. \]
Thus the formula for $B^{2j-1}$ is true.
\begin{rem} The case $k=1$ goes along the same lines as above with the exception when we have that $1+l=n+j+2$. Then $l=n+j+1>n$ and this is not possible. \end{rem}
Now we are ready to finish the proof of the lemma. 
Consider the matrix whose columns are the diagonals of the odd powers of $B$ written from left to right. Observe that in each odd power the main diagonal meets only one of the sub-anti-diagonals that we highlighted above. 
Therefore,  we have that 
\[\left[\text{diag}(B)\, \text{diag}(B^3)\, \ldots \, \text{diag}(B^{2j-1}) \, \ldots \, \text{diag}(B^{2n-1})\right]=\]
\begin{multicols}{2}
for matrices of odd sizes 
\begin{tiny}\[\left[\begin{array}{cccccccccccc}
                                  0&0& 0&0&0&&          0&0&0&1\\
                                  0&0&0&0&0&&           0&1&*&*\\
                                  0&0&0&0&0&&           *&*&*&* \\
                                   &&  \ldots  &&&& \quad &\ldots\\
                                  0&0&0&0&1&&           *&*&*&*  \\
                                  0&0&1&*&*&&             *&*&*&* \\
                                  1&  & *&*&*& \ldots&  *&*&*&* \\
                                  0&1&* &*&*&&            *&*&*&* \\
                                  0&0&0&1&*&&           *&*&*&* \\
                                  && \ldots &&&&& \quad  \ldots \\
                                  0&0&0&0&0&&           *&*&*&* \\
                                  0&0&0&0&0&&           1&*&*&* \\
                                  0&0&0& 0&0&&           0&*&1&* \\
\end{array}\right] ,\] \end{tiny}

and for matrices of even sizes 
 \begin{tiny}\[\left[\begin{array}{cccccccccccc}
                                  0& 0&0&0&&          0&0&0&-1\\
                                  0&0&0&0&&           0&-1&*&*\\
                                  0&0&0&0&&           *&*&*&* \\
                                   &&  \ldots  &&&& \quad &\ldots\\
                                  0&0&0&-1&&           *&*&*&*  \\
                                  0&-1&*&*&&             *&*&*&* \\
                                  -1&* &*&*&&            *&*&*&* \\
                                  0&0&-1&*&&           *&*&*&* \\
                                   \ldots &&&&& \quad  \ldots \\
                                  0&0&0&0&&           *&*&*&* \\
                                  0&0&0&0&&           -1&*&*&* \\
                                  0&0& 0&0&&           0&*&-1&* \\
\end{array}\right] .\] \end{tiny}
\end{multicols}
Since every row of the matrices has a pivot position, we conclude that the determinants of these matrices are not zero and thus the columns of each of them span $\mathbb{Z}^n$. This finishes the proof of the lemma. 
\end{proof}                 
\begin{thm} Let $X=(x_{ij})_{1\leq i, \, j \leq n}$ be a square matrix of size $n$ with indeterminate entries over a field $K$ and $R=K[X] $ be the polynomial ring over $K$ in $\{x_{ij}\, |\, 1\leq i, \, j \leq n\}$. Let \[\Lambda=\left\{(i,j) \, | 1 \leq i\leq n, \, n+1-i \leq j \leq n \right\}\] and \[\Omega=\left\{ (k,l)\, | 1 \leq k\leq n-1, \, 1\leq  l \leq n-k, \, (k,l)\neq (1,1) \right\}.\] 
Then  \[ \mathcal{S}= \left\{ x_{ij},\,  x_{11}-x_{kl} \, |  \, (i,j) \in \Lambda, (k,l) \in \Omega \right \}\]  is a homogeneous system of parameters and hence a regular sequence on $R/(\mathcal{P}(X))$. \end{thm}
\begin{proof}
We prove the lemma by induction on $n$. 

Consider the first few small cases.

Let $n=2$. In this case modulo $(\mathcal{S})$ we have that  \[X= \left[\begin{array}{cc}
                                                   x_{11}&0\\
                                                   0&0 \end{array}\right] \] and \[\mathcal{P}(X)=\det \left[\begin{array}{cc}
                                                                                                                                                     1&x_{11}\\
                                                                                                                                                      1&0 \end{array}\right] =-x_{11}.\] 
                                                                                                                                                      
Let $n=3$. In this case \[X= \left[\begin{array}{ccc}
                                                   x_{11}&x_{11}&0\\
                                                   x_{11}&0&0\\
                                                   0&0&0 \end{array}\right] , \, X^2= \left[\begin{array}{ccc}
                                                                                                                                     2x_{11}^2&x_{11}^2&0\\
                                                                                                                                     x_{11}^2&x_{11}^2&0\\
                                                                                                                                     0&0&0 \end{array}\right] \] and \[\mathcal{P}(X)=\det \left[\begin{array}{ccc}
                                                                                                                                                                                                                                    1&x_{11}&2x_{11}^2\\
                                                                                                                                                                                                                                    1&0&x_{11}^2\\
                                                                                                                                                                                                                                    1&0&0 \end{array}\right] =x_{11}^3. \]
Let $n=4$. In this case \[X= \left[\begin{array}{cccc}
                                                   x_{11}&x_{11}&x_{11}&0\\
                                                   x_{11}&x_{11}&0&0\\
                                                   x_{11}&0&0&0\\
                                                   0&0&0&0 \end{array}\right] , \, X^2= \left[\begin{array}{cccc}
                                                                                                                                     3x_{11}^2&2x_{11}^2&x_{11}^2&0\\
                                                                                                                                     2x_{11}^2&2x_{11}^2&x_{11}^2&0\\
                                                                                                                                     x_{11}^2&x_{11}^2&x_{11}^2&0\\
                                                                                                                                     0&0&0&0 \end{array}\right] , \]

\[X^3= \left[\begin{array}{cccc}
                                           6x_{11}^3&5x_{11}^3&3x_{11}^3&0\\
                                           5x_{11}^3&4x_{11}^3&2x_{11}^3&0\\
                                            3x_{11}^3&2x_{11}^3&x_{11}^3&0\\
                                           0&0&0&0 \end{array}\right] ,                                                                                                                                            
                                                                                                \text{ and } \mathcal{P}(X)=\det \left[\begin{array}{cccc}
                                                                                                                                                1&x_{11}&3x_{11}^2&6x_{11}^3\\
                                                                                                                                                1&x_{11}&2x_{11}^2&4x_{11}^3\\
                                                                                                                                                1&0&x_{11}^2&x_{11}^3\\
                                                                                                                                                1&0&0&0 \end{array}\right] =-x_{11}^6.\]                                                                                                                                                                                                                                                                                                                                                                                          

\begin{claim} $\mathcal{P}(X)=\pm x_{11}^{n(n-1)/2}$ in the quotient ring of $R$ by the ideal generated by \[\{x_{ij}, x_{11}-x_{kl}\, | \, (i,j) \in \Lambda, (k,l) \in \Omega\} \]  \end{claim}

\begin{rem} In our proof of the claim we do not establish exactly the sign of the image of $\mathcal{P}(X)$. We only show that it is either 1 or -1. \end{rem}
Consider the matrices $X$, $X_0$, $\tilde{X}$, and set the elements strictly below the main anti-diagonal of $X_0$ and $x_{nn}$ to 0, that is, \[\tilde{\tilde{X}}= \left[\begin{array}{ccccc|c}
                                                                                                                                 x_{11}&x_{12}&\ldots&x_{1, n-2}&x_{1, n-1}&0\\
                                                                                                                                 x_{21}&x_{22}&\ldots&x_{2, n-2}&0&0\\
                                                                                                                                 &\ldots&\ldots&&0&0\\
                                                                                                                                 x_{n-1,1}&0&\ldots&0&0&0\\
                                                                                                                                 \hline
                                                                                                                                 0&0&\ldots&0&0&0 \end{array}\right] .\]
                                                                                                                                 
Since $P(\tilde{X})=P(X_0)c_{X_0}(x_{nn})$ where $c_{X_0}$ is the characteristic polynomial of $X_0$ (see Lemma~\ref{ColP}), we have that $P(\tilde{\tilde{X}})=P(\tilde{X_0})(-1)^{(n-2)(n-1)/2}\prod_{i=1}^{n-1}x_{i, n-i}$. 
Kill the elements $x_{11}-x_{kl}$ for all $(k,l) \in \Omega$. Then $P(\tilde{X_0})=x_{11}^{(n-1)(n-2)/2}P(Y_0)$, where  
\[Y_0= \left[\begin{array}{ccccc}
                                      1&1&\ldots&1&1\\
                                      1&1&\ldots&1&0\\
                                      &\ldots&\ldots&&\\
                                     1&0&\ldots&0&0\\ \end{array}\right] .\]
Hence $$P(\tilde{\tilde{X}})=x_{11}^{(n-1)(n-2)/2}P(Y_0)(-1)^{(n-2)(n-1)/2}\prod_{i=1}^{n-1}x_{11}=\pm x_{11}^{n(n-1)/2}P(Y_0). $$                                                 
By Lemma~\ref{P(A)=1} we have that $P(Y_0)=\pm1$, which finishes the proof of the claim.

Finally, we have that $R/(\mathcal{P}(X), \mathcal{S})\cong K[x_{11}]/(x_{11}^{n(n-1)/2})$, which has Krull dimension 0. Hence, $\mathcal{S}$ is indeed a system of parameters on $R/(\mathcal{P}(X))$ and, since the ring is a complete intersection, it is also a regular sequence. 
\end{proof}

\section{The variety defined by $\mathcal{P}(X)$ is $F$-pure}
\begin{defn} Let $S$ be a ring with positive prime characteristic $p$. Then $S$ is called $F$-pure if  the Frobenius endomorphism $F: S \rightarrow S$ with $F(s)=s^p$ is pure, that is,  for every $S$-module $M$ we have that $F\otimes_S \text{id}_M: S\otimes M \rightarrow S\otimes M$ is injective. \end{defn} 
Next is Fedder's criterion specialized for hypersurfaces and we use it to prove $F$-purity of $R/(\mathcal{P}(X))$. 
\begin{lem}  (Fedder's criterion, \cite{F83}, Proposition 2.1) \label{FedC}Let $S$ be a polynomial ring over a field $K$ of positive prime characteristic $p$ and let $f \in S$ be a homogeneous polynomial. Then $S/(f)$ is $F$-pure if the polynomial $f^{p-1}$ has a non-zero monomial term in which every indeterminate has degree at most $p-1$.  \end{lem}
We also need the fact that $F$-purity deforms for Gorenstein rings. 
\begin{lem} (\cite{F83}, Theorem 3.4(2)) Let $S$ be a Gorenstein ring and let $x \in S$ be a non-zero divisor. Then if $S/xS$ is $F$-pure, then so is $S$.   \end{lem}
For more examples of $F$-pure rings and other related notions the reader may refer to \cite{F87}, \cite{FW89}, \cite{HR76}, \cite{BH98}. Computer algebra system Macaulay2, \cite{GS}, is a great tool in studying rings and their properties. \\

Let $\mathcal{S}_{\Lambda}=\{ x_{ij} : 1\leq i\leq n,\,  n-i+1\leq j \leq n\}\subseteq  \mathcal{S} $. These are the entries of the matrix $X$ on and below the main anti-diagonal. It has been shown in the previous section that $\mathcal{S}_{\Lambda}$ is part of a homogeneous system of parameters and a regular sequence on $R/(\mathcal{P}(X))$.



\begin{thm} Let $K$ be a field of positive prime characteristic $p$. Then $R/(\mathcal{P}(X))$ is $F$-pure for all $n\geq 1$. \end{thm}
\begin{proof} We use the fact that  $F$-purity deforms for Gorenstein rings, \cite{F83}, and we have that $R/(\mathcal{P}(X))$ is a complete intersection and hence is Gorenstein. Therefore, it is sufficient to show that we have an $F$-pure ring once we take the quotient of the ring $R/(\mathcal{P}(X))$ by the ideal generated by the regular sequence $\mathcal{S}_{\Lambda}$.  

Let us first take a look at what we have for few small values of $n$. 

Case $n=1$ is trivial. $\mathcal{P}(X)=1$ and $R/(\mathcal{P}(X))=0$. 

Case $n=2$: $\mathcal{P}(X)=\det \left[\begin{array}{cc}
                                                   1&x_{11}\\
                                                   1&x_{22} \end{array}\right] =x_{22}-x_{11}$.  Then $R/(\mathcal{P}(X))\cong K[x_{11}, x_{12}, x_{21}]$ is regular and hence $F$-pure. 
                                                   
                                                 
                                                                                                                                                                                                                      
Now we are ready to prove the general statement. We do it by induction on $n$. 

First observe the following: if our statement is true for a fixed $n$, that is, if $R/(\mathcal{P}(X), \mathcal{S}_{\Lambda})$ is $F$-pure, then so is $R/(\mathcal{P}(X), \mathcal{S}_{\Lambda}^0)$, where $$\mathcal{S}_{\Lambda}^0=\{ x_{ij} : 2\leq i\leq n,\,  n-i+2\leq j \leq n\}\subseteq \mathcal{S}_{\Lambda},$$ that is, the entries strictly below the main anti-diagonal. Moreover, if a particular monomial term of $\mathcal{P}(X)$ has a nonzero coefficient in $R/(\mathcal{S}_{\Lambda})$, then it is a monomial term  of $\mathcal{P}(X)$ with a nonzero coefficient in $R/(\mathcal{S}_{\Lambda}^0)$. A partial converse is also true. If $\mathcal{P}(X)$ in $R/(\mathcal{S}_{\Lambda}^0)$ has a nonzero monomial term in the entries of $X$ which are strictly above the main anti-diagonal, then so does $\mathcal{P}(X)$ in $R/(\mathcal{S}_{\Lambda})$. 

The first non-trivial case is $n=3$.  Let $\hat{X}$ be the matrix $X$ modulo the elements of $\mathcal{S}_{\Lambda}^0$, that is, 
                                               \[\hat{X}=\left[\begin{array}{ccc}
                                                                 x_{11}&x_{12}&x_{13}\\
                                                                 x_{21}&x_{22}&0\\
                                                                 x_{31}&0&0 \end{array}\right]\] and \[P(\hat{X})=\det\left[\begin{array}{ccc}
                                                                                                                                                                                      1&x_{11}&x_{11}^2+x_{12}x_{21}x_{13}x_{31}\\
                                                                                                                                                                                      1&x_{22}&x_{22}^2+x_{12}x_{21}\\
                                                                                                                                                                                      1&0& x_{13}x_{31}\end{array}\right]=\]
\[=-x_{11}x_{13}x_{31}+x_{11}x_{12}x_{21}-x_{12}x_{21}x_{22}+x_{11}x_{22}^2-x_{11}^2x_{22}.\]

Since $P(\hat{X})$ has a monomial term $x_{11}x_{12}x_{21}$ with a coefficient $1$ modulo $p$, so does $\mathcal{P}(X)$ in $R/(\mathcal{S}_{\Lambda})$.   

Here is our induction hypothesis: for all $k<n$ we have that $\mathcal{P}(X)^{p-1}$ in $R/(\mathcal{S}_{\Lambda})$ and in $R/(\mathcal{S}_{\Lambda}^0)$ has a monomial term $\prod_{i=1}^{k-1}\prod_{j=1}^{k-i}x_{ij}^{p-1}$ with coefficient $\pm1$ modulo $p$. In other words, this monomial term is the $(p-1)^{\text{st}}$ power of the product of the entries of $X$ strictly above the main anti-diagonal.  The basis of the induction is verified above. 

Now consider the matrices $X_0$, $\tilde{X}$ and $\tilde{\tilde{X}}$. As in the Theorem 1, since $P(\tilde{X})=P(X_0)c_{X_0}(x_{nn})$ where $c_{X_0}$ is the characteristic polynomial of $X_0$ (see Lemma~\ref{ColP}) we have that $P(\tilde{\tilde{X}})=P(\tilde{X_0})(-1)^{(n-2)(n-1)/2}\prod_{i=1}^{n-1}x_{i, n-i}$. 
By induction hypothesis  $P(\tilde{X_0})^{p-1}$ has a monomial term  $\prod_{i=1}^{n-2}\prod_{j=1}^{n-1-i}x_{ij}^{p-1}$ with coefficient $\pm1$.

Hence  $P(\tilde{\tilde{X}})^{p-1}$ has a monomial term $\prod_{i=1}^{n-1}\prod_{j=1}^{ n-i}x_{ij}^{p-1}$ with coefficient $\pm1$.  Therefore, by Fedder's criterion we have that $R/(\mathcal{P}(X), \mathcal{S}_{\Lambda})$ is $F$-pure and thus so is $R/(\mathcal{P}(X))$.                                                                                                                                                                                                 
\end{proof}

The next natural question that one can consider is whether the variety defined by the polynomial $\mathcal{P}(X)$ is $F$-regular. It is certainly true when $n=1$ and $n=2$, but is unknown for larger values of $n$. 

\begin{conj}
Let $K$ be a field of positive prime characteristic $p$. Then $R/(\mathcal{P}(X))$ is $F$-regular for all $n\geq 1$. 
\end{conj}

\section{Acknowledgement}

The author is grateful to Mel Hochster for valuable discussions and comments and thanks the referee for useful suggestions on improving the paper. 
\bibliographystyle{amsalpha}       
\bibliography{mybib}

\end{document}